\documentclass[11pt, oneside]{article}   	
\usepackage{geometry}                		
\usepackage{graphicx}

\usepackage{amssymb,amsmath,amsfonts,bbm,amsthm}
\usepackage{comment}
\usepackage{enumerate}
\usepackage{url}

\newtheorem{theo}{Theorem}[section]

\newtheorem{prop}[theo]{Proposition}

\newtheorem{lemma}[theo]{Lemma}

\newcommand{\Cc}{\mathbbm{C}}

\newcommand{\Qq}{\mathbbm{Q}}

\newcommand{\Aa}{\mathbbm{A}}
\newcommand{\Zz}{\mathbbm{Z}}

\DeclareMathOperator{\tr}{tr}

\DeclareMathOperator{\Ind}{Ind}

\DeclareMathOperator{\GL}{GL}
\DeclareMathOperator{\MM}{M}

\DeclareMathOperator{\SL}{SL}

\DeclareMathOperator{\cInd}{c-Ind}

\DeclareMathOperator{\supp}{supp}
\DeclareMathOperator{\cond}{cond}

\title{Computing the local $2$-component of a non-selfdual automorphic representation of $\GL_3$}
\author{YAMAMOTO Hirofumi }
\date{}							

\begin{document}
\maketitle

\begin{abstract}
In this paper, we explicitly determine the local $2$-adic component of a non-selfdual automorphic representation $\Pi$ of $\GL_3$ constructed by van Geemen and Top.
We prove that $\Pi_2$ is a parabolically induced representation of $\GL_3(\Qq_2)$ given by $\Pi_2 = \Ind_P^{\GL_3(\Qq_2)}(\pi\boxtimes \chi)$, where $P$ is the standard parabolic subgroup of $\GL_3$ with Levi subgroup $\GL_2 \times \GL_1$, $\chi$ is an unramified character of $\Qq_2^\times$ satisfying $\chi(2) = -2\sqrt{-1}$, and $\pi$ is a supercuspidal representation of $\GL_2(\Qq_2)$. Furthermore, we describe $\pi$ explicitly as a compactly induced representation $\pi = \cInd_{J_\alpha}^{\GL_2(\Qq_2)} \Lambda$ and determine the representation $\Lambda$ explicitly.
The proof relies on explicit computations of Hecke eigenvalues using computer calculations. The automorphic representation $\Pi$ is realized in the cuspidal cohomology of the congruence subgroup $\Gamma_0(128) \subset \SL_3(\Zz)$. By computing the Hecke eigenvalues of an associated Hecke eigenvector, we are able to uniquely identify the local structure of $\Pi_2$.
As an application, we obtain an explicit description of the $2$-adic local component of the Galois representation $\rho_{\mathrm{vGT},\ell}$ associated with $\Pi$.
\end{abstract}

\section{Introduction}
In this paper, we compute the local component $\Pi_2$ at $2$ of the non-selfdual cuspidal automorphic representation $\Pi$ of $\GL_3(\Aa_\Qq)$ constructed in \cite{vanGreemen1994}. It is proved in \cite[Theorem $1.2$]{Ito2018GaloisRA} that this representation is associated with the Galois representation $\rho_{\mathrm{vGT}, l}$ also constructed in \cite{vanGreemen1994}, by decomposing the transcendental part of the $l$-adic cohomology of a projective smooth surface $S_2$ whose affine model is 
\[
t^2 = xy(x^2-1)(y^2-1)(x^2-y^2+2xy).
\]

For any prime $p\ge 3$, the local component $\Pi_p$ of this representation is unramified, and its local L-factor can be computed by using the algorithm given in \cite{AshRudilph1979}. Using the same algorithm, we can compute several Hecke eigenvalues of $\Pi_2$ which enables us to identify $\Pi_{2}$ explicitly and obtain the following theorem.
\begin{theo}
Let $\psi$ be an additive character of $\Qq_2$ such that $\psi|_{2\Zz_2}$ is trivial and $\psi|_{\Zz_2}$ is non-trivial. We have
\[\Pi_2 = \Ind_P^{\GL_3(\Qq_2)}(\pi\boxtimes \chi),\] 
where $\Ind$ is the unnormalized induction, $\chi$ is the unramified character of $\Qq_2^\times$ with $\chi(2) = -2\sqrt{-1}$ and $\pi$ is the supercuspidal representation of $\GL_2(\Qq_2)$ given by 
\[
\pi = \cInd_{J_\alpha}^{\GL_2(\Qq_2)} \Lambda.
\]
Here $\cInd$ is the compact induction, $\alpha =\frac{1}{8} \left(\begin{matrix}&1\\-2&\end{matrix}\right)$, $E = \Qq_2(\alpha)$ is a quadratic extension of $\Qq_2$, $\mathfrak{P} = \{ \left(\begin{matrix} a&b\\c&d\end{matrix}\right) \in \MM_2(\Zz_2)\mid a,c,d\equiv 0 \bmod 2\}$, $U_\mathfrak{U}^3 = 1 + \mathfrak{P}^3$, $J_\alpha=E^\times U_\mathfrak{U}^3$, $\psi_\alpha$ is the character of $U_\mathfrak{U}^3 $ such that $\psi_\alpha(x) = \psi \circ \tr(\alpha (x-1))$ for $x \in U_\mathfrak{U}^3$ and $\Lambda$ is the character of $J_\alpha$ determined by 
\[
\begin{cases}
\Lambda(\left(\begin{matrix}a&\\&a\end{matrix}\right)) = \chi^{-1}(a) &a\in\Qq_2^\times,\\
\Lambda(\left(\begin{matrix}&1\\-2&\end{matrix}\right)) = \frac{1+\sqrt{-1}}{2},\\
\Lambda(\left(\begin{matrix}1&1\\-2&1\end{matrix}\right)) = \psi\left(\frac{1}{2}\right),\\
\Lambda|_{U_\mathfrak{U}^3} = \psi_\alpha.
\end{cases}
\]
\end{theo}

By this theorem and \cite[Theorem $1.2$]{Ito2018GaloisRA}, we obtain that 
\[
\rho_{\mathrm{vGT}, l}|_{W_{\Qq_2}} = \rho_{\pi}\otimes  |\cdot|^{1/2}\oplus  \chi\otimes  |\cdot|^{-1},
\]
where $W_{\Qq_2}$ is the Weil group of $\Qq_2$,  $\rho_\pi$ is the irreducible $2$-dimensional representation corresponding to the supercuspidal representation $\pi$ by the local Langlands correspondence for $\GL_2$, $|\cdot|$ is the normalized absolute value on $\Qq_2^\times$, and $|\cdot|$ and $ \chi $ are identified with the characters of $W_{\Qq_2}$ by the local class field theory.

We prove this theorem in three steps. First, we compute the eigenvalues of $\Pi_2$ for several Hecke operators by using a computer. Second, we show that $\Pi_2$ is not supercuspidal by using the fact that the eigenvalue of the Hecke operator associated with $\left(\begin{matrix}2&&\\&2&\\&&1\end{matrix}\right)$ is nonzero. Third, we show that $\Pi_2$ is of the form $\Ind_P^{\GL_3(\Qq_2)} (\pi \boxtimes \chi)$, where $\pi$ is a supercuspidal representation of $\GL_2(\Qq_2)$ whose conductor is $2^7$ and $\chi$ is an unramified character of $\Qq_2^\times$, by using the eigenvalues of the Hecke operators associated with $\left(\begin{matrix}1&&\\64&1&\\&&1\end{matrix}\right)$ and $\left(\begin{matrix}2&&\\&1&\\&&1\end{matrix}\right)$. Finally, we identify $\pi$ by using calculation of the eigenvalues of the Hecke operators associated with $\left(\begin{matrix}36&1&\\128&4&\\&&4\end{matrix}\right)$, $\left(\begin{matrix}8&1&\\128&8&\\&&8\end{matrix}\right)$ and $\left(\begin{matrix}&1&\\-128&&\\&&1\end{matrix}\right)$.

{\bf Acknowledgements:} The author is grateful to his advisor Yoichi Mieda for a lot of guidance, useful comments and encouragement.

\section{Modular forms and Hecke operators}
For an integer $N\ge 1$, we define the subgroup of $\Gamma_0(N)$ of $ \SL_3(\Zz)$ by
\[
\Gamma_0(N)  = \{\left(\begin{matrix} *&*&*\\a&*&*\\b&*&*\end{matrix}\right) \in \SL_3(\Zz) \mid a\equiv b\equiv 0 \bmod N\}
\]
and define the finite set $\mathbbm{P}^2(\Zz/N)$ 
\begin{align*}
\mathbbm{P}^2(\Zz/N) &= \{(x,y,z) \in (\Zz/N)^3 \mid x\Zz/N+y\Zz/N+z\Zz/N =\Zz/N\} / (\Zz/N)^\times \\
&\cong  \SL_3(\Zz) / \Gamma_0(N) .
\end{align*}

The automorphic representation of $\mathrm{GL}_{3}$ that we study is defined by a cuspidal
cohomology class in \(H^{3}(\Gamma_{0}(128),\mathbb{C})\). We denote the subspace of all cuspidal
classes by \(H^{3}_{!}(\Gamma_{0}(N),\mathbb{C})\). By \cite[Lemma~3.5]{ASH1984412}, there exists a surjective map
\[
H_{3}(\Gamma_{0}(N),\mathbb{C}) \longrightarrow H^{3}_{!}(\Gamma_{0}(N),\mathbb{C}),
\]
where \(H_{3}(\Gamma_{0}(N),\mathbb{C})\) denotes the dual space of \(H^{3}(\Gamma_{0}(N),\mathbb{C})\).

\begin{theo}[{\cite[Theorem $3.2$ and Proposition $3.12$]{ASH1984412}}] \label{thm2.1}
There is a canonical isomorphism between 
\[
H_3(\Gamma_0(N), \Cc)
\]
and the vector space of maps $f \colon \mathbbm{P}^2(\Zz/N)\to \Cc$ which satisfy
\begin{enumerate}
\item $f(x:y:z) = -f(-y:x:z)$,
\item $f(x:y:z) = f(z:x:y)$,
\item $f(x:y:z) + f(-y:x-y:z) + f(y-x:-x:y) = 0$.
\end{enumerate}
\end{theo}

By \cite[Section $4$]{ASH1984412}, for $\alpha \in \GL_3(\Qq)$ we can define the Hecke operator 
\[
T^{\GL_3}_\alpha \colon H^3(\Gamma_0(N)) \to H^3(\Gamma_0(N)),
\]
which corresponds to the double coset $\Gamma_0(N)\alpha\Gamma_0(N)$.
Also, by \cite[Section $2.5$]{vanGreemen1994}, there exists a Hecke eigenvector $f\in H^3(\Gamma_0(128))$ such that for all odd prime $p$,
\[
T^{\GL_3}_{\left(\begin{smallmatrix}p&&\\&1&\\&&1\end{smallmatrix}\right)} f= a_pf
\]
with $a_p \in \Zz[\sqrt{-1}]$ and $a_3 = 1+2\sqrt{-1}$. Moreover, $f$ is in $H^3_!(\Gamma_0(128), \Cc)$ by \cite[Section $1.2$]{Ito2018GaloisRA}.

Using explicit computer computation, we obtained the following results. First, by Theorem \ref{thm2.1}, we verified that $H_3(\Gamma_0(128), \Cc)$ is $58$-dimensional. Moreover, using the algorithm described in \cite[Section $2.10$]{MR1474576}, we found that the eigenspace 
\[
\{ f' \in H^3(\Gamma_0(128)) \mid T^{\GL_3}_{\left(\begin{smallmatrix}3&&\\&1&\\&&1\end{smallmatrix}\right)} f' =(1+2\sqrt{-1})f'\}
\]
is $1$-dimensional. Applying the same algorithm, we obtained the following Hecke eigenvalues of $f$: 

\begin{align}
	T^{\GL_3}_{\left(\begin{smallmatrix}2&&\\&2&\\&&1\end{smallmatrix}\right)}f &= 2\sqrt{-1}f \label{hecke1},\\
	T^{\GL_3}_{\left(\begin{smallmatrix}2&&\\&1&\\&&1\end{smallmatrix}\right)}f &= 0 \label{hecke2},\\
	T^{\GL_3}_{\left(\begin{smallmatrix}1&&\\64&1&\\&&1\end{smallmatrix}\right)}f&=-f\label{hecke3},\\
	T^{\GL_3}_{\left(\begin{smallmatrix}36&1&\\128&4&\\&&4\end{smallmatrix}\right)}f&=8f \label{hecke6},\\
	T^{\GL_3}_{\left(\begin{smallmatrix}8&1&\\128&8&\\&&8\end{smallmatrix}\right)}f&=32f \label{hecke7},\\
	T^{\GL_3}_{\left(\begin{smallmatrix}&1&\\-128&&\\&&1\end{smallmatrix}\right)}f&=(8-8\sqrt{-1})f\label{hecke5},\\
	T^{\GL_3}_{\left(\begin{smallmatrix}a&&\\&a&\\&&a\end{smallmatrix}\right)}f&=f &&(a \in \Qq^{\times})\label{centchar}.
\end{align}
The source code of the computer programs used in these computations is available at the following URL: \url{https://github.com/yamamoto277/culc_of_heckeeigenval/blob/11ad32c85444c0207343af5e596e6d6b74fcec55/culc_of_heckeeigenval.py}.

Let $\Pi$ be the $2$-component of the cuspidal automorphic representation corresponding to $f$. 
We define the compact open subgroups $K^{\GL_3}$ and $K'^{\GL_3}$ by
\begin{align*}
K^{\GL_3} &:= \left\{A = (a_{ij})\in\GL_3(\Zz_2)\mid a_{21} \equiv a_{31} \equiv 0 \bmod 128\right\},\\
K'^{\GL_3} &:= \left\{A = (a_{ij})\in\GL_3(\Zz_2)\mid a_{21} \equiv a_{31} \equiv 0 \bmod 64\right\}.
\end{align*}

\begin{prop}\label{prop22}
We have $\dim \Pi^{K^{\GL_3}}=1$ and $ \Pi^{K'^{\GL_3}}=0$.
\end{prop}

\begin{proof}
The claim $\dim \Pi^{K^{\GL_3}}=1$ follows from the fact that the eigenspace $\{ f' \in H^3(\Gamma_0(128)) \mid T^{\GL_3}_{\left(\begin{smallmatrix}3&&\\&1&\\&&1\end{smallmatrix}\right)} f' =(1+2\sqrt{-1})f'\}$ is $1$-dimensional. Assume $ \Pi^{K'^{\GL_3}}\not=0$. Since $ \Pi^{K'^{\GL_3}} \subset \Pi^{K^{\GL_3}}$, we have $ \Pi^{K'^{\GL_3}} = \Pi^{K^{\GL_3}}$. Therefore $f$ belongs to the image of $ H^3(\Gamma_0(64)) \to  H^3(\Gamma_0(128))$, which contradicts to Equation (\ref{hecke3}). Therefore we obtain $\Pi^{K'^{\GL_3}}=0$.
\end{proof}
For $g \in \GL_3(\Qq_2)$, we define the Hecke operator $T^\Pi_g$ acting on $\Pi^{K^{\GL_3}}$ by 
\begin{align*}
T^\Pi_g v = \sum_{i} \Pi(g_i)v &&v\in \Pi^{K^{\GL_3}} ,
\end{align*}
where $K^{\GL_3}gK^{\GL_3} = \coprod_i g_iK^{\GL_3}$. By Proposition \ref{prop22}, for $v \in  \Pi^{K^{\GL_3}}$ and $\alpha \in \GL_3(\Qq)$ such that $\alpha \in \GL_3(\Zz_p)$ for all primes $p>2$, we have 
\[
T^\Pi_\alpha v = a_\alpha v,
\]
where $a_\alpha$ is defined by $T^{\GL_3}_\alpha f = a_\alpha f$.  

\section{Determination of $\Pi$}
Fix $v\in \Pi^{K^{\GL_3}}$.
We have the double coset decomposition
\begin{align*}
K^{\GL_3}\left(\begin{matrix}2&&\\&2&\\&&1\end{matrix}\right)K^{\GL_3} =\coprod_{i,j\in \{0,1\}} \left(\begin{matrix}2&&i\\&2&j\\&&1\end{matrix}\right)K^{\GL_3} \amalg \coprod_{i\in\{0,1\}}\left(\begin{matrix}2&i&\\&1&\\&&2\end{matrix}\right)K^{\GL_3}.
\end{align*}
Let us define a Hecke operator $T$ acting on $v$ by
\begin{align*}
Tv :&= T^\Pi_{\left(\begin{smallmatrix}2&&\\&2&\\&&1\end{smallmatrix}\right)}v\\
&=\sum_{x,y\in \{0,1\}} \Pi(\left(\begin{matrix}2&&x\\&2&y\\&&1\end{matrix}\right))v + \sum_{x\in\{0,1\}}\Pi(\left(\begin{matrix}2&x&\\&1 &\\&&2\end{matrix}\right))v.
\end{align*}
Then, by Equation ($\ref{hecke1}$), we obtain $Tv= 2\sqrt{-1}v$.

\begin{lemma}\label{lemma31}
For each positive integer $n$, we have
\begin{align*}
&T^nv \\
&=\sum_{i=0}^n\Pi(\left(\begin{matrix}2^{n}&&\\&2^{n-i}&\\&&2^i\end{matrix}\right))\sum_{j=0}^{\min(i, n-i)} a_{n,i,j}\sum_{\substack{ 0\le x < 2^i\\0\le y<2^{n-i}\\ 0\le z< 2^{n-i-j}}}\Pi(\left(\begin{matrix}1&2^{-i}x&2^{i-n}y+2^{j-n}xz\\&1&2^{i+j-n}z\\&&1\end{matrix}\right))v,
\end{align*}
where $a_{n,i,j}$'s are constants depending only on $n$, $i$ and $j$.
\end{lemma}
\begin{proof}
We prove the claim by induction. When $n=1$, then
{\small
\begin{align*}
Tv &= \sum_{x,y\in \{0,1\}} \Pi(\left(\begin{matrix}2&&x\\&2&y\\&&1\end{matrix}\right))v + \sum_{x\in\{0,1\}}\Pi(\left(\begin{matrix}2&x&\\&1&\\&&2\end{matrix}\right))v\\
&= \Pi(\left(\begin{matrix}2&&\\&2&\\&&1\end{matrix}\right))\sum_{\substack{ 0\le y<2\\ 0\le z< 2}}\Pi(\left(\begin{matrix}1&&2^{-1}y\\&1&2^{-1}z\\&&1\end{matrix}\right))v + \Pi(\left(\begin{matrix}2&&\\&1&\\&&2\end{matrix}\right))\sum_{\substack{ 0\le x < 2}}\Pi(\left(\begin{matrix}1&2^{-1}x&\\&1&\\&&1\end{matrix}\right))v.
\end{align*}
}
For integers $i$, $j$, $x$, $y$ and $z$ with $0\le i \le n$, $0\le j \le i$, $0\le x < 2^i$, $0\le y < 2^{n-i}$, $0\le z < 2^{n-i+j}$, and integers $0\le y'$, $z'\le1$, we have
\begin{align*}
&\left(\begin{matrix}2&&y'\\&2&z'\\&&1\end{matrix}\right) \cdot \left(\begin{matrix}2^{n}&&\\&2^{n-i}&\\&&2^i\end{matrix}\right)\left(\begin{matrix}1&2^{-i}x&2^{i-n}y+2^{j-n}xz\\&1&2^{i+j-n}z\\&&1\end{matrix}\right)\\
&=\left(\begin{matrix}2^{n+1}&&\\&2^{n+1-i}&\\&&2^i\end{matrix}\right)\left(\begin{matrix}1&2^{-i}x &2^{i-(n+1)}\left(2y + y'-xz'\right)+2^{j-(n+1)}x\left(2z+2^{i-j}z'\right) \\&1&2^{i+j-(n+1)}\left(2z+2^{i-j}z'\right)\\&&1\end{matrix}\right).
\end{align*}

Also, for integers $i$, $j$, $x$, $y$ and $z$ with $0\le i \le n$, $0\le j \le i$, $0\le x < 2^i$, $0\le y < 2^{n-i}$, $0\le z < 2^{n-i-j}$, and integers $0\le x'\le1$, we have
\begin{align*}
&\left(\begin{matrix}2&x'&\\&1&\\&&2\end{matrix}\right) \cdot \left(\begin{matrix}2^{n}&&\\&2^{n-i}&\\&&2^i\end{matrix}\right)\left(\begin{matrix}1&2^{-i}x&2^{i-n}y+2^{j-n}xz\\&1&2^{i+j-n}z\\&&1\end{matrix}\right)\\
&=\left(\begin{matrix}2^{n+1}&&\\&2^{n-i}&\\&&2^{i+1}\end{matrix}\right)\left(\begin{matrix}1&2^{-i-1}(2x + x') &2^{i-n}y+ 2^{j-(n+1)}(2x+x')z\\&1&2^{i+j-n}z\\&&1\end{matrix}\right).
\end{align*}
Let us denote 
\[
S(n,i,j) = \sum_{\substack{ 0\le x < 2^i\\0\le y<2^{n-i}\\ 0\le z< 2^{n-i-j}}}\Pi(\left(\begin{matrix}2^{n}&&\\&2^{n-i}&\\&&2^i\end{matrix}\right)\left(\begin{matrix}1&2^{-i}x&2^{i-n}y+2^{j-n}xz\\&1&2^{i+j-n}z\\&&1\end{matrix}\right))v.
\]
Then, when $i>j$, we obtain
{\tiny
\begin{align*}
TS(n,i,j) &= \sum_{y', z' \in \{0,1\}} \sum_{\substack{ 0\le x < 2^i\\0\le y<2^{n-i}\\ 0\le z< 2^{n-i-j}}}\Pi(\left(\begin{matrix}2^{n+1}&&\\&2^{n+1-i}&\\&&2^i\end{matrix}\right)\left(\begin{matrix}1&2^{-i}x &\frac{2^{i}\left(2y + y'-xz'\right)+2^{j+1}x\left(z+2^{i-j-1}z'\right)}{2^{n+1}} \\&1&2^{i+j-n}\left(z+2^{i-j-1}z'\right)\\&&1\end{matrix}\right))v\\
&+ \sum_{x' \in \{0,1\}} \sum_{\substack{ 0\le x < 2^i\\0\le y<2^{n-i}\\ 0\le z< 2^{n-i-j}}}\Pi(\left(\begin{matrix}2^{n+1}&&\\&2^{n-i}&\\&&2^{i+1}\end{matrix}\right)\left(\begin{matrix}1&2^{-i-1}(2x + x') &\frac{2^{i+1}y+ 2^{j}(2x+x')z}{2^{n+1}}\\&1&2^{i+j-n}z\\&&1\end{matrix}\right))v\\
&= 2S(n+1, i, j+1) + S(n+1, i+1, j),
\end{align*}
}
and when $i=j$, we have
{\tiny
\begin{align*}
TS(n,i,j)  &= \sum_{y', z' \in \{0,1\}} \sum_{\substack{ 0\le x < 2^i\\0\le y<2^{n-i}\\ 0\le z< 2^{n-i-j}}}\Pi(\left(\begin{matrix}2^{n+1}&&\\&2^{n+1-i}&\\&&2^i\end{matrix}\right)\left(\begin{matrix}1&2^{-i}x &\frac{2^{i}\left(2y + y'-xz'\right)+2^{j}x\left(2z+z'\right)}{2^{n+1}} \\&1&2^{i+j-n-1}\left(2z+z'\right)\\&&1\end{matrix}\right))v\\
&+ \sum_{x' \in \{0,1\}} \sum_{\substack{ 0\le x < 2^i\\0\le y<2^{n-i}\\ 0\le z< 2^{n-i-j}}}\Pi(\left(\begin{matrix}2^{n+1}&&\\&2^{n-i}&\\&&2^{i+1}\end{matrix}\right)\left(\begin{matrix}1&2^{-i-1}(2x + x') &\frac{2^{i+1}y+ 2^{j}(2x+x')z}{2^{n+1}}\\&1&2^{i+j-n}z\\&&1\end{matrix}\right))v\\
&= S(n+1, i, j) + S(n+1, i+1, j).
\end{align*}

}

Both are in the form required by the lemma, so we can express $T^n$ as in this lemma.
\end{proof}

\begin{prop}\label{prop32}
	The representation $\Pi$ is not supercuspidal.
\end{prop}
\begin{proof}

	Suppose that $\Pi$ is supercuspidal.
	Then, since the Jacquet modules of $\Pi$ are zero, there exists a positive integer $M$ such that for any integer $m> M$, 
	\begin{align*}
		\sum_{0\le x, y <2^{m}}\Pi(\left(\begin{matrix}1&&2^{-m}x\\&1&2^{-m}y\\&&1\end{matrix}\right))v = 0.
	\end{align*}
	Also, there exists a positive integer $M'$ such that for any integer $m'> M'$ and for any integer $m \le M$, 
	\begin{align*}
		\sum_{0\le x ,y<2^{m'}}\Pi(\left(\begin{matrix}1&2^{-m'}x&2^{-m'}y\\&1&\\&&1\end{matrix}\right))\sum_{0\le z <2^{m}}\Pi(\left(\begin{matrix}1&&\\&1&2^{-m}z\\&&1\end{matrix}\right))v = 0.
	\end{align*}
	
Let $n$ be an integer which satisfies $n > 2(M + M')$. For integers $0\le i \le n$, $0\le j \le \min(i, n-i)$ with $n-i-j > M$, we have
{\small
\begin{align*}
S(n,i,j)&=\Pi(\left(\begin{matrix}2^n&&\\&2^i&\\&&2^{n-i}\end{matrix}\right))\sum_{\substack{ 0\le x < 2^i\\0\le y<2^{n-i}\\ 0\le z< 2^{n-i-j}}}\Pi(\left(\begin{matrix}1&2^{-i}x&2^{i-n}y+2^{j-n}xz\\&1&2^{i+j-n}z\\&&1\end{matrix}\right))v\\
&=\sum_{\substack{ 0\le x < 2^i\\0\le y '<2^{j}}}\Pi(\left(\begin{matrix}2^n&2^{n-i}x&2^{i}y'\\&2^i&\\&&2^{n-i}\end{matrix}\right))\sum_{\substack{0\le y <2^{n-i-j}\\0\le z<2^{n-i-j}}}\Pi(\left(\begin{matrix}1&&2^{i+j-n}y\\&1&2^{i+j-n}z\\&&1\end{matrix}\right))v\\
&=0.
\end{align*}
}
For integers $0\le i \le n$, $0\le j \le \min(i, n-i)$ with $n-i-j \le M$, we have 
\begin{align*}
i &\ge n-j-M\\
& \ge n- \min(i, n-i) - M \\
&\ge n-\frac{n}{2} -M \\
&> (M+M') -M = M'.
\end{align*}
Therefore, we have
{\small
\begin{align*}
&\sum_{x'=0}^{2^n-1}\Pi(\left(\begin{matrix}1&&x'\\&1&\\&&1\end{matrix}\right))S(n,i,j)\\
&=\sum_{x'=0}^{2^n-1}\Pi(\left(\begin{matrix}1&&x'\\&1&\\&&1\end{matrix}\right))\Pi(\left(\begin{matrix}2^n&&\\&2^i&\\&&2^{n-i}\end{matrix}\right))\sum_{\substack{ 0\le x < 2^i\\0\le y<2^{n-i}\\ 0\le z< 2^{n-i-j}}}\Pi(\left(\begin{matrix}1&2^{-i}x&2^{i-n}y+2^{j-n}xz\\&1&2^{i+j-n}z\\&&1\end{matrix}\right))v\\
&=\sum_{ \substack{0 \le x'' < 2^{n-i}\\0\le y <2^{n-i}}}\Pi(\left(\begin{matrix}2^n&&2^{i}y + 2^nx''\\&2^i&\\&&2^{n-i}\end{matrix}\right))\\
&\cdot\sum_{0\le x , x'< 2^i}\Pi(\left(\begin{matrix}1&2^{-i}x&2^{-i}x'\\&1&\\&&1\end{matrix}\right))\sum_{0\le z<2^{n-i-j}}\Pi(\left(\begin{matrix}1&&\\&1&2^{i+j-n}z\\&&1\end{matrix}\right))v\\
&=0. 
\end{align*}
}
Therefore, by Lemma \ref{lemma31} and Equation (\ref{hecke1}), we have
\begin{align*}
2^n (2\sqrt{-1})^nv &= \sum_{x'=0}^{2^n-1}\Pi(\left(\begin{matrix}1&&x'\\&1&\\&&1\end{matrix}\right))T^nv\\
 &= \sum_{x'=0}^{2^n-1}\Pi(\left(\begin{matrix}1&&x'\\&1&\\&&1\end{matrix}\right))\sum_{i=0}^n\sum_{j=0}^{\min(i, n-i)} a_{n,i,j}S(n,i,j)\\
 &=\sum_{\substack{0\le i \le n \\ 0\le j \le \min(i, n-i)\\n-i-j > M}}  a_{n,i,j}\sum_{x'=0}^{2^n-1}\Pi(\left(\begin{matrix}1&&x'\\&1&\\&&1\end{matrix}\right))S(n,i,j) \\
 &+ \sum_{\substack{0\le i \le n \\ 0\le j \le \min(i, n-i)\\n-i-j \le M}}  a_{n,i,j}\sum_{x'=0}^{2^n-1}\Pi(\left(\begin{matrix}1&&x'\\&1&\\&&1\end{matrix}\right))S(n,i,j) \\
 &= 0,
 \end{align*}
 which is a contradiction. So $\Pi$ is not supercuspidal. 
\end{proof}

For a character $\chi$ of $\Qq_2^\times$, we define $\cond(\chi) = \min\{n \in \Zz_{\ge0} \mid \chi(1+2^nx ) = 1$ for all $x\in\Zz_2\}$. Also we define the subgroup $K'_n$ of $\GL_2(\Zz_2)$ for $n\in\Zz_{\ge 0}$ by 
\begin{align*}
K'_n = \{\left(\begin{matrix}a&b\\c&d\end{matrix}\right) \in \GL_2(\Zz_2) \mid c\equiv 0, \,d\equiv 1 \bmod 2^n\}.
\end{align*}
For an irreducible infinite-dimensional representation $\pi$ of $\GL_2(\Qq_2)$, we define 
\begin{align*}
\cond(\pi) = \min\{n\in\Zz_{\ge 0} \mid \pi^{K'_n} \not =0\}.
\end{align*}
By \cite[Theorem $1$]{MR337789}, this is well-defined.

\begin{prop}\label{prop33} There exist an irreducible supercuspidal representation $\pi$ of $\GL_2(\Qq_2)$ with $\cond(\pi) = 7$ and an unramified character $\chi$ of $\Qq_2^\times$ with $\chi(2) = -2\sqrt{-1}$ such that
\begin{align*}
\Pi = \Ind_P^{\GL_3}(\pi\boxtimes \chi).
\end{align*}
Here  $\Ind$ is the unnormalized induction and $P$ is a parabolic subgroup of $\GL_3$ of the following form:
\[
	P = \left\{(a_{ij})_{i,j}\in\GL_3\mid a_{31} = a_{32} = 0\right\}.
\]
\end{prop}

\begin{proof}
From Proposition \ref{prop32}, there exist an irreducible infinite-dimensional representation $\pi$ of $\GL_2(\Qq_2)$ and a character $\chi$ of $\Qq_2^\times$ such that $\Pi$ is a subrepresentation of a representation $\Ind_P^{\GL_3}(\pi\boxtimes \chi)$.
 
We first determine a set of representatives for $P(\Qq_2)\backslash \GL_3(\Qq_2) / K^{\GL_3}$. It is clear that
\begin{align*}
P(\Qq_2)\backslash \GL_3(\Qq_2) / K^{\GL_3}&\cong P(\Zz_2)\backslash \GL_3(\Zz_2) / K^{\GL_3}\\
&\cong P(\Zz/2^7)\backslash \GL_3(\Zz/2^7) / Q(\Zz/2^7),\\
\end{align*}
where $Q = \{(a_{ij})_{i,j}\mid a_{21} = a_{31} = 0\}$.
Let us define
\begin{align*}
 K_1^{\GL_3} &:= \{(a_{ij})_{i,j} \in \GL_3(\Zz_2)\mid a_{21} \equiv a_{31} \equiv 0 \bmod 2\},\\
 s_1 &:=\left(\begin{matrix}0&1&0\\1&0&0\\0&0&1\end{matrix}\right), s_2 :=\left(\begin{matrix}0&0&1\\0&1&0\\1&0&0\end{matrix}\right).
\end{align*}
Let $B$ denote the Borel subgroup of upper-triangular matrices of $\GL_3$. Then by the Bruhat decomposition, we have
{\small
\begin{align*}
\GL_3(\Zz/2\Zz) &= Q(\Zz/2\Zz) \amalg B(\Zz/2\Zz)s_1Q(\Zz/2\Zz)\amalg B(\Zz/2\Zz)s_2Q(\Zz/2\Zz). 
\end{align*}
}
Also, we have
\begin{align*}
 B(\Zz/2\Zz)s_1Q(\Zz/2\Zz) &= \coprod_{x\in\{0,1\}}s_1\left(\begin{matrix}1&&\\x&1&\\&&1\end{matrix}\right)Q(\Zz/2\Zz),\\
  B(\Zz/2\Zz)s_2Q(\Zz/2\Zz)&= \coprod_{x,y\in\{0,1\}}s_2\left(\begin{matrix}1&&\\x&1&\\y&&1\end{matrix}\right)Q(\Zz/2\Zz).
 \end{align*}
 Hence we obtain the following complete system of representatives for $\GL_3(\Zz_2) / K^{\GL_3}_1 \cong \GL_3(\Zz/2\Zz) / Q(\Zz/2\Zz)$:
\[
1,\,s_1\left(\begin{matrix}1&&\\x&1&\\&&1\end{matrix}\right), \, s_2\left(\begin{matrix}1&&\\x&1&\\y&&1\end{matrix}\right), 
\]
where $x$, $y\in\{0,1\}$.
In addition, since
\[
K_1^{\GL_3}/K^{\GL_3} = \coprod_{\substack{x, y\in\Zz/2^7\Zz\\x, y \in 2\Zz}}\left(\begin{matrix}1&&\\x&1&\\y&&1\end{matrix}\right),
\]
a complete set of representatives for $\GL_3(\Zz_2) / K^{\GL_3}$ is given by
\begin{align*}
\left(\begin{matrix}1&&\\x&1&\\y&&1\end{matrix}\right), & &x, y\in2\Zz, 0\le x, y< 2^7,\\
s_1\left(\begin{matrix}1&&\\x&1&\\y&&1\end{matrix}\right), && x\in \Zz, y\in2\Zz, 0\le x, y< 2^7,\\
s_2\left(\begin{matrix}1&&\\x&1&\\y&&1\end{matrix}\right), && x, y\in\Zz, 0\le x, y< 2^7.
\end{align*}
Using the following matrix identities:
\begin{align*}
s_2\left(\begin{matrix}1&&\\x&1&\\y&&1\end{matrix}\right) &= \left(\begin{matrix}1&&y\\&1&x\\&&1\end{matrix}\right) s_2,\\
\left(\begin{matrix}1&&\\x&1&\\y&&1\end{matrix}\right) &= \left(\begin{matrix}1&&\\x&1&\\&&1\end{matrix}\right) \left(\begin{matrix}1&&\\&1&\\y&&1\end{matrix}\right),\\
\left(\begin{matrix}1&&\\&1&\\x\cdot 2^i&&1\end{matrix}\right) &= \left(\begin{matrix}x^{-1}&&\\&1&\\&&1\end{matrix}\right)\left(\begin{matrix}1&&\\&1&\\ 2^i&&1\end{matrix}\right)\left(\begin{matrix}x&&\\&1&\\&&1\end{matrix}\right) & x\in\Zz_2^\times, i \in \Zz, 1\le i\le 6,
\end{align*}
we obtain the following complete system of representatives for $P(\Zz_2) \backslash \GL_3(\Zz_2) / K^{\GL_3}$ : 
\[
1, s_2, M_i = \left(\begin{matrix}1&&\\&1&\\ 2^i&&1\end{matrix}\right), 1\le i\le 6.
\]

Let $f\in\Pi^{K^{\GL_3}} \subset \Ind_P^{\GL_3}(\pi\boxtimes \chi)^{K^{\GL_3}}$ be a nonzero vector.
First, by Equation (\ref{hecke3}) we have
\begin{align*}
 -f(s_2)&=T_{\left(\begin{smallmatrix}1&&\\64&1&\\&&1\end{smallmatrix}\right)}^\Pi f(s_2) \\
 &= \Pi( \left(\begin{matrix}1&&\\64&1&\\&&1\end{matrix}\right) )f(s_2) + \Pi( \left(\begin{matrix}1&&\\&1&\\64&&1\end{matrix}\right) )f(s_2)+ \Pi( \left(\begin{matrix}1&&\\64&1&\\64&&1\end{matrix}\right) )f(s_2)\\
&=f(\left(\begin{matrix}1&&\\&1&64\\&&1\end{matrix}\right)s_2) + f(\left(\begin{matrix}1&&64\\&1&\\&&1\end{matrix}\right)s_2) + f(\left(\begin{matrix}1&&64\\&1&64\\&&1\end{matrix}\right)s_2)\\
&=3f(s_2),
\end{align*}
so we have $f(s_2) = 0$.

By Equation (\ref{hecke2}), we have
\begin{align*}
0 &= T_{\left(\begin{smallmatrix}2&&\\&1&\\&&1\end{smallmatrix}\right)}^\Pi f(1) \\
&= \sum_{0\le j, k \le1}\Pi(\left(\begin{matrix}2&j&k\\&1&\\&&1\end{matrix}\right)) f(1) \\
&= \sum_{0\le j, k \le1}f(\left(\begin{matrix}2&j&k\\&1&\\&&1\end{matrix}\right))\\
&=2\sum_{0\le j \le1}\pi(\left(\begin{matrix}2&j\\&1\end{matrix}\right)) f(1),
\end{align*}
so we have
\begin{align*}
\sum_{0\le j \le1}\pi(\left(\begin{matrix}2&j\\&1\end{matrix}\right)) f(1) = 0. \stepcounter{equation}\tag{\theequation}\label{prop330}
\end{align*}

We show that $f(M_i) =0$ for $ 1 \le i \le 6$. Assume that there exists an integer $i_0$ with $1\le i_0\le 6$ such that $f(M_{i_0})\not=0$.
For $a_{ij} \in \Zz_2$, $1\le i$, $j\le 3$ with $a_{11}$, $a_{22} \in \Zz_2^\times$ we have
\begin{align*}
&\pi(\left(\begin{matrix}a_{11}-2^{i_0}a_{13}&a_{12}\\2^7a_{21}-2^{i_0}a_{23}&a_{22}\end{matrix}\right))\chi(a_{11}+2^{7-i_0}a_{31})f(M_{i_0})\\
&= f(\left(\begin{matrix}a_{11}-2^{i_0}a_{13}&a_{12}&a_{13}\\2^7a_{21}-2^{i_0}a_{23}&a_{22}&a_{23}\\0&0&a_{11}+2^{7-i_0}a_{31}\end{matrix}\right)M_{i_0})\\
&= \Pi(\left(\begin{matrix}a_{11}&a_{12}&a_{13}\\2^7a_{21}&a_{22}&a_{23}\\2^7a_{31}&-2^{i_0}a_{12}&a_{11}-2^{i_0}a_{13}+2^{7-i_0}a_{31}\end{matrix}\right))f(M_{i_0})\\
&=f(M_{i_0}).\stepcounter{equation}\tag{\theequation}\label{eq331}
\end{align*}
This implies that $f(M_{i_0}) \in \pi^{K'_{i_0}}$, $\cond(\pi) \le i_0$ and $\cond(\chi) \le 7-i_0$.

Assume that $i_0=6$. By Equation (\ref{hecke2}) and the inequalities $\cond(\pi) \le 6$ and $\cond(\chi)\le 1$, we have
\begin{align*}
0 &= T_{\left(\begin{smallmatrix}2&&\\&1&\\&&1\end{smallmatrix}\right)}^\Pi f(M_{5}) \\
&= \sum_{0\le j, k \le1}\Pi(\left(\begin{matrix}2&j&k\\&1&\\&&1\end{matrix}\right)) f(M_{5}) \\
&= \sum_{0\le j, k \le1}f(\left(\begin{matrix}2&j-\frac{2^{5}jk}{1+2^{5}k}&k\\&1&\\&&1+2^{5}k\end{matrix}\right)M_{6}\left(\begin{matrix}\frac{1}{1+2^{5}k}&&\\&1&\\&\frac{2^{5}j}{1+2^{5}k}&1\end{matrix}\right))\\
&=\sum_{0\le j \le1}\sum_{0\le k \le 1}\chi(1+2^{5}k)\pi(\left(\begin{matrix}2&j-\frac{2^{5}jk}{1+2^{5}k}\\&1\end{matrix}\right)) f(M_{6})\\
&=\sum_{0\le j \le1}\pi(\left(\begin{matrix}2&j\\&1\end{matrix}\right))\sum_{0\le k \le 1}\pi(\left(\begin{matrix}1&-\frac{2^{4}jk}{1+2^{5}k}\\&1\end{matrix}\right))\left( f(M_{6})\right)\\
&=2\sum_{0\le j \le1}\pi(\left(\begin{matrix}2&j\\&1\end{matrix}\right)) f(M_{6}).
\end{align*}
The last identity follows from $f(M_6) \in  \pi^{K'_{6}}$. 
Therefore we have $\sum_{0\le j \le1}\pi(\left(\begin{matrix}2&j\\&1\end{matrix}\right)) f(M_{6})=0$.

By Equation (\ref{hecke1}) and the inequality $\cond(\pi) \le 6$, we have 
\begin{align*}
&2\sqrt{-1}f(M_{6})\\
&=T_{\left(\begin{smallmatrix}2&&\\&2&\\&&1\end{smallmatrix}\right)}^\Pi f(M_{6}) \\
&=\sum_{0\le j, k \le1}\Pi(\left(\begin{matrix}2&&j\\&2&k\\&&1\end{matrix}\right)) f(M_{6}) + \sum_{0\le j \le1}\Pi(\left(\begin{matrix}2&j&\\&1&\\&&2\end{matrix}\right)) f(M_{6}) \\
&= \sum_{0\le j, k \le1}f(M_{6}\left(\begin{matrix}2&&j\\&2&k\\&&1\end{matrix}\right)) + \sum_{0\le j \le1}f(M_{6}\left(\begin{matrix}2&j&\\&1&\\&&2\end{matrix}\right))\\
&=\sum_{0\le j, k \le1}f(\left(\begin{matrix}\frac{2}{1+2^{6}j} &0&\frac{j}{1+2^{6}j}\\-\frac{2^{7}k}{1+2^{6}j}&2&\frac{k}{1+2^{6}j}\\0&0&1\end{matrix}\right)\left(\begin{matrix}1&&\\&1&\\2^7&&1+2^{6}j\end{matrix}\right)) \\
&+ \sum_{0\le j \le1}f(\left(\begin{matrix}2&j&\\&1&\\&&2\end{matrix}\right)M_{6}\left(\begin{matrix}1&&\\&1&\\&2^{5}j&1\end{matrix}\right))\\
&= 4\pi(\left(\begin{matrix}2&\\&2\end{matrix}\right))f(1) + \chi(2)\sum_{0\le j \le1}\pi(\left(\begin{matrix}2&j\\&1\end{matrix}\right)) f(M_{6}) \\
&= 4\pi(\left(\begin{matrix}2&\\&2\end{matrix}\right))f(1) .\stepcounter{equation}\tag{\theequation}\label{prop332}
\end{align*}

By Equation (\ref{hecke3}) and the inequality $\cond(\pi) \le 6$, we have
\begin{align*}
 -f(M_6)&=T_{\left(\begin{smallmatrix}1&&\\64&1&\\&&1\end{smallmatrix}\right)}^\Pi f(M_6) \\
&= f(M_6\left(\begin{matrix}1&&\\64&1&\\&&1\end{matrix}\right)) + f(M_6\left(\begin{matrix}1&&\\&1&\\64&&1\end{matrix}\right)) + f(M_6\left(\begin{matrix}1&&\\64&1&\\64&&1\end{matrix}\right))\\
&= f(\left(\begin{matrix}1&&\\64&1&\\&&1\end{matrix}\right)M_6) + f(\left(\begin{matrix}1&&\\&1&\\128&&1\end{matrix}\right)) + f(\left(\begin{matrix}1&&\\64&1&\\&&1\end{matrix}\right)\left(\begin{matrix}1&&\\&1&\\128&&1\end{matrix}\right))\\
&= f(M_6) + 2f(1),
\end{align*}
so $f(M_6) = -f(1)$.
By Equations (\ref{hecke1}), (\ref{prop330}) and (\ref{prop332}), we have
\begin{align*}
&2\sqrt{-1}f(1)\\
&=T_{\left(\begin{smallmatrix}2&&\\&2&\\&&1\end{smallmatrix}\right)}^\Pi f(1) \\
&=\sum_{0\le j, k \le1}\Pi(\left(\begin{matrix}2&&j\\&2&k\\&&1\end{matrix}\right)) f(1) + \sum_{0\le j \le1}\Pi(\left(\begin{matrix}2&j&\\&1&\\&&2\end{matrix}\right)) f(1) \\
&= 4\pi(\left(\begin{matrix}2&\\&2\end{matrix}\right))f(1) + \chi(2)\sum_{0\le j \le1}\pi(\left(\begin{matrix}2&j\\&1\end{matrix}\right)) f(1)\\
&= 2\sqrt{-1}f(M_6) = -2\sqrt{-1}f(1).
\end{align*}
Hence, we have $f(1) = 0$ and then $f(M_6) = -f(1) = 0$, which is a contradiction. Therefore we have $f(M_6) = 0$ and $i_0 \not=6$.

By Equation (\ref{hecke3}) and the inequality $\cond(\pi) \le i_0$, we have
{\small
\begin{align*}
 -f(M_{i_0})&=T_{\left(\begin{smallmatrix}1&&\\64&1&\\&&1\end{smallmatrix}\right)}^\Pi f(M_{i_0}) \\
&= f(M_{i_0}\left(\begin{matrix}1&&\\64&1&\\&&1\end{matrix}\right)) + f(M_{i_0}\left(\begin{matrix}1&&\\&1&\\64&&1\end{matrix}\right)) + f(M_{i_0}\left(\begin{matrix}1&&\\64&1&\\64&&1\end{matrix}\right))\\
&=f(\left(\begin{matrix}1&&\\64&1&\\&&1\end{matrix}\right)M_{i_0}) + f(\left(\begin{matrix}1&&\\&1&\\&&1+2^{6-i_0}\end{matrix}\right)M_{i_0}\left(\begin{matrix}1&&\\&1&\\&&\frac{1}{1+2^{6-i_0}}\end{matrix}\right))  \\
& + f(\left(\begin{matrix}1&&\\64&1&\\&&1+2^{6-i_0}\end{matrix}\right)M_{i_0}\left(\begin{matrix}1&&\\&1&\\&&\frac{1}{1+2^{6-i_0}}\end{matrix}\right))\\
&=(1+2\chi(1+2^{6-i_0}))f(M_{i_0}).
\end{align*}
}Note that $1+2^{6-i_0} \in \Zz_2^\times$ since $1 \le i_0 \le 5$. Therefore, $\chi(1+2^{6-i_0}) = -1$ and then $\cond(\chi) = 7-i_0$. In particular, $i_0$ is uniquely determined by $\chi$. Therefore we have $f(M_j)=0$ for $1\le j \le 5$, $j\not=i_0$. 
Since the central character $\omega_\Pi$ of $\Pi$ is trivial, the central character $\omega_\pi$ of $\pi$ is equal to $\omega_\Pi \chi^{-1} = \chi^{-1}$. If $1\le i_0 \le 3$, we obtain
\begin{align*}
\cond (\chi )= \cond(\omega_\pi) \le \cond(\pi) \le i_0 < 7-i_0 =\cond (\chi),
\end{align*}
which is a contradiction. Therefore we have  $4\le i_0 \le 5$. 

By Equation (\ref{hecke2}) and the inequalities $\cond(\pi) \le i_0$ and $\cond(\chi) = 7-i_0 \le i_0-1$, we have
\begin{align*}
0 &= T_{\left(\begin{smallmatrix}2&&\\&1&\\&&1\end{smallmatrix}\right)}^\Pi f(M_{i_0-1}) \\
&= \sum_{0\le j, k \le1}\Pi(\left(\begin{matrix}2&j&k\\&1&\\&&1\end{matrix}\right)) f(M_{i_0-1}) \\
&= \sum_{0\le j, k \le1}f(\left(\begin{matrix}2&j-\frac{2^{i_0-1}jk}{1+2^{i_0-1}k}&k\\&1&\\&&1+2^{i_0-1}k\end{matrix}\right)M_{i_0}\left(\begin{matrix}\frac{1}{1+2^{i_0-1}k}&&\\&1&\\&\frac{2^{i_0-1}j}{1+2^{i_0-1}k}&1\end{matrix}\right))\\
&=\sum_{0\le j \le1}\sum_{0\le k \le 1}\chi(1+2^{i_0-1}k)\pi(\left(\begin{matrix}2&j-\frac{2^{i_0-1}jk}{1+2^{i_0-1}k}\\&1\end{matrix}\right)) f(M_{i_0})\\
&=\sum_{0\le j \le1}\pi(\left(\begin{matrix}2&j\\&1\end{matrix}\right))\sum_{0\le k \le 1}\pi(\left(\begin{matrix}1&-\frac{2^{i_0-2}jk}{1+2^{i_0-1}k}\\&1\end{matrix}\right))\left( f(M_{i_0})\right)\\
&=2\sum_{0\le j \le1}\pi(\left(\begin{matrix}2&j\\&1\end{matrix}\right)) f(M_{i_0}).
\end{align*}
The last identity follows from $f(M_{i_0}) \in  \pi^{K'_{i_0}}$. Therefore we obtain $\sum_{0\le j \le1}\pi(\left(\begin{matrix}2&j\\&1\end{matrix}\right)) f(M_{i_0})=0$.

By Equation (\ref{hecke1}) and the inequality $\cond(\pi) \le i_0$, we have 
\begin{align*}
&2\sqrt{-1}f(M_{i_0})\\
&=T_{\left(\begin{smallmatrix}2&&\\&2&\\&&1\end{smallmatrix}\right)}^\Pi f(M_{i_0}) \\
&=\sum_{0\le j, k \le1}\Pi(\left(\begin{matrix}2&&j\\&2&k\\&&1\end{matrix}\right)) f(M_{i_0}) + \sum_{0\le j \le1}\Pi(\left(\begin{matrix}2&j&\\&1&\\&&2\end{matrix}\right)) f(M_{i_0}) \\
&= \sum_{0\le j, k \le1}f(M_{i_0}\left(\begin{matrix}2&&j\\&2&k\\&&1\end{matrix}\right)) + \sum_{0\le j \le1}f(M_{i_0}\left(\begin{matrix}2&j&\\&1&\\&&2\end{matrix}\right))\\
&=\sum_{0\le j, k \le1}f(\left(\begin{matrix}\frac{2}{1+2^{i_0}j} &0&\frac{j}{1+2^{i_0}j}\\-\frac{2^{i_0+1}k}{1+2^{i_0}j}&2&\frac{k}{1+2^{i_0}j}\\0&0&1\end{matrix}\right)M_{i_0+1}\left(\begin{matrix}1&&\\&1&\\&&1+2^{i_0}j\end{matrix}\right)) \\
&+ \sum_{0\le j \le1}f(\left(\begin{matrix}2&j&\\&1&\\&&2\end{matrix}\right)M_{i_0}\left(\begin{matrix}1&&\\&1&\\&2^{i_0-1}j&1\end{matrix}\right))\\
&= 4\pi(\left(\begin{matrix}2&\\&2\end{matrix}\right))f(M_{i_0+1}) + \chi(2)\sum_{0\le j \le1}\pi(\left(\begin{matrix}2&j\\&1\end{matrix}\right)) f(M_{i_0}) \\
&= 4\pi(\left(\begin{matrix}2&\\&2\end{matrix}\right))f(M_{i_0+1})=0,
\end{align*}
which is a contradiction.
Therefore we have $f(M_{i}) = 0$ for $1 \le i \le 6$, and also we have $f(1) \not=0$.

For $a_{ij} \in \Zz_2$, $1\le i, j\le 3$ with $a_{11}$, $a_{22}$, $a_{33} \in \Zz_2^\times$ we have
\begin{align*}
\pi(\left(\begin{matrix}a_{11}&a_{12}\\2^7a_{21}&a_{22}\end{matrix}\right))\chi(a_{33})f(1) = f(\left(\begin{matrix}a_{11}&a_{12}&a_{13}\\2^7a_{21}&a_{22}&a_{23}\\&&a_{33}\end{matrix}\right))=f(1).
\end{align*}
Therefore, $\cond(\pi)\le 7$ and $\chi$ is an unramified character. 

We show that $\Pi$ is not a discrete series representation. If $\Pi$ is a discrete series representation, then by the unramifiedness of $\chi$, $\Pi$ is the twist by an unramified character of the Steinberg representation. In particular, by \cite[Proposition $2.4$]{MR4522693}, $\Pi$ has a nonzero $K^{\GL_3}_2$-fixed vector, where  $K_2^{\GL_3} := \{(a_{ij})_{i,j} \in \GL_3(\Zz_2)\mid a_{21} \equiv a_{31} \equiv 0 \bmod 4\}$. This contradicts to Proposition \ref{prop22}, so $\Pi$ is not a discrete series representation. Since $\Pi$ is tempered by \cite[Corollary $1.4$]{Ito2018GaloisRA}, we can take $\pi$ and $\chi$ so that $\Pi = \Ind_P^{\GL_3}(\pi \boxtimes \chi)$.

We show $\cond(\pi) = 7$. By Equation (\ref{hecke3}) we have
\begin{align*}
 -f(1)&=T_{\left(\begin{smallmatrix}1&&\\64&1&\\&&1\end{smallmatrix}\right)}^\Pi f(1) \\
&= f(\left(\begin{matrix}1&&\\64&1&\\&&1\end{matrix}\right)) + f(\left(\begin{matrix}1&&\\&1&\\64&&1\end{matrix}\right)) + f(\left(\begin{matrix}1&&\\64&1&\\64&&1\end{matrix}\right))\\
&= \Pi(\left(\begin{matrix}1&&\\64&1&\\&&1\end{matrix}\right))f(1) + f(M_6)+\Pi(\left(\begin{matrix}1&&\\64&1&\\&&1\end{matrix}\right))f(M_6)\\
&= \pi(\left(\begin{matrix}1&\\64&1\end{matrix}\right))f(1) + f(M_6)+\pi(\left(\begin{matrix}1&\\64&1\end{matrix}\right))f(M_6)\\
&=  \pi(\left(\begin{matrix}1&\\64&1\end{matrix}\right))f(1).
\end{align*}
In particular, we have $f(1) \notin \pi^{K_6'}$. If $\cond(\pi) \le 6$, then we have a nonzero vector $v' \in \pi^{K_6'}$. Let $f' \in \Pi^{K^{\GL_3}} = \Ind(\pi\boxtimes\chi)^{K^{\GL_3}}$ be the vector such that $f'(1) = v'$, $f'(s_2) = 0$ and $f'(M_i) = 0$ for $1\le i \le 6$. By Proposition \ref{prop22}, $f'$ is a constant multiple of $f$. So we have $v' = f'(1) \notin  \pi^{K_6'}$, which is a contradiction. Therefore, we have $\cond(\pi) = 7$.

We compute $\chi(2)$.
By Equations (\ref{hecke1}), (\ref{centchar}), (\ref{prop330}) and $\omega_\pi = \chi^{-1}$, we obtain
\begin{align*}
	2\sqrt{-1}f(1)&=T_{\left(\begin{smallmatrix}2&&\\&2&\\&&1\end{smallmatrix}\right)}^\Pi f(1) \\
	&=\sum_{0\le j, k \le1}\Pi(\left(\begin{matrix}2&&j\\&2&k\\&&1\end{matrix}\right)) f(1) + \sum_{0\le j \le1}\Pi(\left(\begin{matrix}2&j&\\&1&\\&&2\end{matrix}\right)) f(1) \\
	&=\sum_{0\le j, k \le1}\pi(\left(\begin{matrix}2&\\&2\end{matrix}\right)) f(1) + \sum_{0\le j \le1}\pi(\left(\begin{matrix}2&j\\&1\end{matrix}\right)) f(1) \\
	&= 4(\chi(2))^{-1}f(1).
\end{align*}
This implies that $\chi(2) = -2\sqrt{-1}$.

Finally we prove that $\pi$ is supercuspidal.
Since $\chi$ is unramified, $\omega_\pi=  \chi^{-1}$ is also unramified. By \cite[Proposition $2.1.2$]{shimidtgl2} and the subsequent paragraph, if $\pi$ is a subquotient of an induced representation $\Ind_{P'}^{\GL_2}(\chi' \boxtimes (\chi')^{-1}\omega_\pi)$, where $P'$ is a minimal paraboric subgroup of $\GL_2$ and $\chi'$ is a character of $\Qq_2^\times$, then $\cond(\pi) = 1$ or $2\cond(\chi')$. Therefore $\pi$ is supercuspidal. 
\end{proof}


\section{Determination of $\pi$}
Let us fix an additive character $\psi$ of $\Qq_2$ such that $\psi|_{2\Zz_2}$ is trivial and $\psi|_{\Zz_2}$ is non-trivial. We define
\begin{align*}
\mathfrak{U} &:= \{ \left(\begin{matrix} a&b\\c&d\end{matrix}\right) \in \MM_2(\Zz_2)\mid c\equiv 0 \bmod 2\},\\
\mathfrak{P}^n &:=  \left(\begin{matrix}&1\\2&\end{matrix}\right)^n\mathfrak{U} &&\text{for }n \in \Zz,\\
U_\mathfrak{U}^n &:= 1 + \mathfrak{P}^n &&\text{for }n\ge1.
\end{align*}

For an element $\alpha \in \mathfrak{P}^{-5}\setminus  \mathfrak{P}^{-4}$, we define $E = \Qq_2(\alpha) \subset \MM_2(\Qq_2)$, which is a ramified quadratic extension of $\Qq_2$. We also define
\begin{align*}
J_\alpha&:=E^\times U_\mathfrak{U}^3,\\
\psi_\alpha(x) &:= \psi \circ \tr(\alpha (x-1)) \text{ for } x \in U_\mathfrak{U}^3.
\end{align*}
We denote by $C(\psi_\alpha, \mathfrak{U})$ the set of equivalence classes of irreducible representations $\Lambda$ of the group $J_\alpha$ such that $\Lambda|_{U_\mathfrak{U}^{3}} $ is a multiple of $\psi_\alpha$. By \cite[$15.6$, Proposition $1$]{02870c4051e64c73b17d25171df70483}, every $\Lambda \in C(\psi_\alpha, \mathfrak{U})$ has dimension $1$.

Let $\pi$ be the supercuspidal representation of $\GL_2(\Qq_2)$ in Proposition \ref{prop33}. By \cite[Appendix $3.8$]{BM02}, there exist $\alpha \in \mathfrak{P}^{-5}$ and $\Lambda \in C(\psi_\alpha, \mathfrak{U})$ such that
\begin{align*}
\pi = \cInd_{J_\alpha}^{\GL_2(\Qq_2)} \Lambda,
\end{align*}
where $\cInd$ is the compact induction. Note that $\Lambda(\left(\begin{matrix}a&\\&a\end{matrix}\right)) = \omega_\pi(a) = \chi^{-1}(a)$. In particular, $\Lambda|_{\Zz_2^\times}$ is trivial.

To determine $\pi$, we firstly determine $\alpha$. By \cite[$13.5$, Proposition]{02870c4051e64c73b17d25171df70483}, $ \alpha + \mathfrak{P} ^{-4}$ contains no nilpotent element.  Hence we have
\[
\alpha \in \{\frac{1}{8}\left(\begin{matrix}2a&b\\2c&2d\end{matrix}\right) \in M_2(\Qq_2)\mid b,c \in \Zz_2^\times, a, d \in \Zz_2\}.
\]

By \cite[$15.5$, Induction Theorem]{02870c4051e64c73b17d25171df70483}, the equivalence classes of cuspidal representations correspond to the conjugacy classes of $(J_\alpha, \Lambda)$, $\Lambda \in C(\psi_\alpha, \mathfrak{U})$. For any element $\alpha = \frac{1}{8}\left(\begin{matrix}2a&b\\2c&2d\end{matrix}\right)$ in $ \mathfrak{P}^{-5}$, we have
\begin{align*}
\left(\begin{matrix}-2ad+bc&a\\0&c\end{matrix}\right)^{-1}\frac{1}{8}\left(\begin{matrix}2a&b\\2c&2d\end{matrix}\right)\left(\begin{matrix}-2ad+bc&a\\0&c\end{matrix}\right)=\frac{1}{8}\left(\begin{matrix}0&1\\ -4ad+2bc &2(a+d) \end{matrix}\right).
\end{align*}
Therefore we may assume $\alpha = \alpha(D, t) := \frac{1}{8}\left(\begin{matrix}0&1\\ 2D &2t \end{matrix} \right)$ for some $D\in\Zz_2^\times$ and $t \in \Zz_2$. 

\begin{lemma}\label{lemmapsi}
For any $D$, $D' \in  \Zz_2^\times$ and $t$, $t' \in \Zz_2$ with $D - D' \in 4\Zz_2$ and $t-t'\in 2\Zz_2$, we have $\psi_{\alpha(D,t)} = \psi_{\alpha(D', t')}$ and $J_{\alpha(D,t)} = J_{\alpha(D', t')}$.
\end{lemma}
\begin{proof}
For $x = \left(\begin{matrix}4x_1+1&2x_2\\4x_3&4x_4+1\end{matrix}\right) \in U_\mathfrak{U}^{3}$ with $x_1$, $x_2$, $x_3$, and $x_4\in\Zz_2$, we have
\[
\tr (\alpha(D,t)(x-1)) = \frac{1}{2}x_3 +\frac{1}{2}Dx_2 + tx_4.
\]
Therefore for any $D$, $D' \in  \Zz_2^\times$ and $t$, $t' \in \Zz_2$ with $D - D' \in 4\Zz_2$ and $t-t'\in 2\Zz_2$, we have $\psi_{\alpha(D,t)} = \psi_{\alpha(D', t')}$.

By \cite[$15.1$, Intertwining Theorem]{02870c4051e64c73b17d25171df70483}, $J_{\alpha(D,t)} $ is the set of intertwiners of $\psi_{\alpha(D,t)} $, and $J_{\alpha(D',t')} $ is the set of intertwiners of $\psi_{\alpha(D',t')} $. Therefore,  $\psi_{\alpha(D,t)} = \psi_{\alpha(D', t')}$ implies $J_{\alpha(D,t)} = J_{\alpha(D', t')}$.
 \end{proof}
 By Lemma \ref{lemmapsi}, we may assume that $D\in\{1, -1\}$ and $t\in\{0,1\}$.
As we saw in the proof of Proposition \ref{prop33}, $\omega_\pi = \chi^{-1}$ is unramified. Since $\psi_{\alpha(1,1)}(\left(\begin{matrix}5&\\&5\end{matrix}\right)) = \psi_{\alpha(-1,1)}(\left(\begin{matrix}5&\\&5\end{matrix}\right)) = \psi(1) =-1$, we have $t=0$. 
 
\begin{lemma}
For any $\alpha = \alpha(D,0)$ with $D\in\{-1,1\}$ and any $\Lambda \in C(\psi_\alpha, \mathfrak{U})$ whose central character is unramified, there exist $F \in \left(\cInd_{J_\alpha}^{\GL_2(\Qq_2)} \Lambda\right)^{K_7}$ such that $F(\left(\begin{matrix}8&\\&1\end{matrix}\right)) = 1$ and $\supp F \subset J_\alpha \left(\begin{matrix}8&\\&1\end{matrix}\right)K_7$.
\end{lemma}
\begin{proof}
First we determine $K_7 \cap \left(\begin{matrix}8&\\&1\end{matrix}\right)^{-1}J_\alpha  \left(\begin{matrix}8&\\&1\end{matrix}\right)$. Let $k = \left(\begin{matrix} a&b\\128c&d\end{matrix}\right) \in K_7$. Then we have $a$, $d \in \Zz_2^\times$ and $b$, $c \in \Zz_2$. 
Also we have 
 \begin{align*}
 E^\times = \{ \left(\begin{matrix}&1\\2D&\end{matrix}\right)^n \mid n\in\Zz \} \times  \left\{ \left(\begin{matrix}x&y\\2Dy&x\end{matrix}\right) \mid x \in \Zz_2^\times, y\in \Zz_2 \right\} .
 \end{align*}
 Since $\det \left( \left(\begin{matrix}8&\\&1\end{matrix}\right)k  \left(\begin{matrix}8&\\&1\end{matrix}\right)^{-1}\right)  \in \Zz_2^\times $ and $U_\mathfrak{U}^3\subset \GL_2(\Zz_2)$, if $k \in  \left(\begin{matrix}8&\\&1\end{matrix}\right)^{-1}J_\alpha  \left(\begin{matrix}8&\\&1\end{matrix}\right)$, then there exist $x \in \Zz_2^\times$, $y\in \Zz_2$ and $r$, $s$, $t$, $u \in \Zz_2$ such that 
\begin{align*}
\left(\begin{matrix}8&\\&1\end{matrix}\right)k  \left(\begin{matrix}8&\\&1\end{matrix}\right)^{-1} &=\left(\begin{matrix}x&y\\2Dy&x\end{matrix}\right)\left(\begin{matrix}1+4r&2s\\4t &1+4u\end{matrix}\right).
\end{align*}
Then we have
\begin{align*}
\left(\begin{matrix} a&8b\\16c&d\end{matrix}\right) &= \left(\begin{matrix}x + 4(rx+ty) & 2sx + (4u+1)y\\ 4tx + 2D(4r+1)y&x + 4(ux+Dsy)\end{matrix}\right).
\end{align*}
In particular we have $ a\equiv d \bmod 4$. On the other hand, if $ a\equiv d \bmod 4$, then we have
 \begin{align*}
 \left(\begin{matrix} a&8b\\16c&d\end{matrix}\right)  =   \left(\begin{matrix} a&\\&a\end{matrix}\right)   \left(\begin{matrix} 1&8a^{-1}b\\16a^{-1}c&a^{-1}d\end{matrix}\right) \in J_\alpha.
 \end{align*}
Therefore we have 
\begin{align*}
K_7 \cap \left(\begin{matrix}8&\\&1\end{matrix}\right)^{-1}J_\alpha  \left(\begin{matrix}8&\\&1\end{matrix}\right) = \left\{\left(\begin{matrix}a&b\\128c&d\end{matrix}\right) \mid b,c \in \Zz_2, a,d\in\Zz_2^\times, a\equiv d \bmod 4\right\}.
\end{align*}

We define a vector $F  \in \left(\cInd_{J_\alpha}^{\GL_2(\Qq_2)} \Lambda\right)^{K_7}$ as follows:
\begin{align*}
\begin{cases}
F(p\left(\begin{matrix}8&\\&1\end{matrix}\right)k ) = \Lambda(p) & \text{$p\in J_\alpha$, $ k \in K_7$,}\\
F(g) = 0 & \text{$g \in \GL_2(\Qq_2) \setminus J_\alpha \left(\begin{matrix}8&\\&1\end{matrix}\right) K_7$}.
\end{cases}
\end{align*}
We show that $F$ is well-defined. Assume that $p$, $p' \in J_\alpha$ and $k$, $k' \in K_7$ satisfy $p\left(\begin{matrix} 8&\\&1\end{matrix}\right)k =  p'\left(\begin{matrix} 8&\\&1\end{matrix}\right)k'$. 
Then $kk'^{-1} =  \left(\begin{matrix} 8&\\&1\end{matrix}\right)^{-1}p^{-1}p'\left(\begin{matrix} 8&\\&1\end{matrix}\right)$ belongs to $ K_7 \cap \left(\begin{matrix}8&\\&1\end{matrix}\right)^{-1}J_\alpha  \left(\begin{matrix}8&\\&1\end{matrix}\right)$. 
We write $kk'^{-1} =  \left(\begin{matrix} a&b\\128c&d\end{matrix}\right)$. 
Then, since $\Lambda|_{\Zz_2^\times}$ is trivial, we have
\begin{align*}
\Lambda(p')&= \Lambda(p\left(\begin{matrix}8&\\&1\end{matrix}\right)\left(\begin{matrix}a&b\\128c&d\end{matrix}\right)\left(\begin{matrix}8&\\&1\end{matrix}\right)^{-1})\\
&=\Lambda(p\left(\begin{matrix}a&\\&a\end{matrix}\right)\left(\begin{matrix}1&8a^{-1}b\\16a^{-1}c&a^{-1}d\end{matrix}\right))\\
&=\Lambda(p)\Lambda(\left(\begin{matrix}a&\\&a\end{matrix}\right))\psi_\alpha(\left(\begin{matrix}1&8a^{-1}b\\16a^{-1}c&a^{-1}d\end{matrix}\right))\\
&= \Lambda(p)\psi(\tr\frac{1}{8}\left(\begin{matrix}&1\\2D&\end{matrix}\right)\left(\begin{matrix}1&8a^{-1}b\\16a^{-1}c&a^{-1}d\end{matrix}\right))\\
&=\Lambda(p) \psi(2a^{-1}(bD+c))\\
&=\Lambda(p).
\end{align*}
Therefore $F$ is well-defined.
 \end{proof}

\begin{lemma}\label{lemma21} We have 
\[
K^{\GL_3} \left(\begin{matrix}36&1&\\128&4&\\&&4\end{matrix}\right)K^{\GL_3} = \coprod_{\substack{r,s,t,u\in\{-1,0,1,2\}\\ ru-st \equiv 1 \bmod 4}}\left(\begin{matrix}4&u&-s\\128r&4&0\\128t&0&4\end{matrix}\right)K^{\GL_3} .
\]
\end{lemma}
\begin{proof} We have
\begin{align*}
 \left(\begin{matrix}36&1&\\128&4&\\&&4\end{matrix}\right)  &=  \left(\begin{matrix}4&1&\\128&4&\\&&4\end{matrix}\right) \left(\begin{matrix}\frac{-1}{7}&&\\\frac{256}{7}&1&\\&&1\end{matrix}\right),
 \end{align*} 
 and for any $\left(\begin{matrix}a&b&c\\128x&r&s\\128y&t&u\end{matrix}\right) \in K_7$, we have
\begin{align*}
\left(\begin{matrix}a&b&c\\128x&r&s\\128y&t&u\end{matrix}\right) &= \left(\begin{matrix}1&&\\&r&s\\&t&u\end{matrix}\right)\left(\begin{matrix}a&&\\&1&\\&&1\end{matrix}\right)\left(\begin{matrix}1&a^{-1}b&a^{-1}c\\ \frac{128(ux-sy)}{ru-st}&1&\\ \frac{128(-tx+ry)}{ru-st}&&1\end{matrix}\right). \stepcounter{equation}\tag{\theequation}\label{decompofK7}
\end{align*}
For $w$, $x$, $y$, $z\in\Zz_2$, we have
\[
\left(\begin{matrix}1&w&x\\ 128y&1&\\128 z&&1\end{matrix}\right) \left(\begin{matrix}4&1&\\128&4&\\&&4\end{matrix}\right)=\left(\begin{matrix}4&1&\\128&4&\\&&4\end{matrix}\right)\left(\begin{matrix}1+\frac{-32w+32y}{7}&\frac{-w+8y}{7}&\frac{-x}{7}\\\frac{1024w-128y}{7}&1+\frac{32w-32y}{7}&\frac{32x}{7}\\128z&32z&1\end{matrix}\right),
\]
and for $a\in\Zz_2^\times$ and $\left(\begin{matrix}r&s\\t&u\end{matrix}\right)\in\GL_2(\Zz_2)$, we have
\begin{align*}
 \left(\begin{matrix}a&&\\ &r&s\\ &t&u\end{matrix}\right) \left(\begin{matrix}4&1&\\128&4&\\&&4\end{matrix}\right) &= \left(\begin{matrix}4&\frac{au}{ru-st}&-\frac{as}{ru-st}\\128a^{-1}r&4&0\\128a^{-1}t&0&4\end{matrix}\right)\left(\begin{matrix}a&&\\ &r&s\\ &t&u\end{matrix}\right). \\
 \end{align*}

For $\left(\begin{matrix}r&s\\t&u\end{matrix}\right)$, $\left(\begin{matrix}r'&s'\\t'&u'\end{matrix}\right) \in \GL_2(\Zz_2)$, we have
\begin{align*}
&\left(\begin{matrix}4&u&-s\\128r&4&\\128t&&4\end{matrix}\right)^{-1}\left(\begin{matrix}4&u'&-s'\\128r'&4&\\128t'&&4\end{matrix}\right)\\
&=\frac{1}{1-8(ru-st)}\left(\begin{matrix}1-8(r'u-st')&\frac{u'-u}{4}&\frac{s-s'}{4}\\ 128\left(\frac{r'-r}{4} + 2s(r't-rt')\right) &1-8(ru'-st)&8r(s'-s)\\128\left(\frac{t'-t}{4} + 2u(r't-rt')\right)&8t(u-u')&1-8(ru-s't)\end{matrix}\right).
\end{align*}
This implies that $\left(\begin{matrix}4&u&-s\\128r&4&\\128t&&4\end{matrix}\right)K^{\GL_3} = \left(\begin{matrix}4&u'&-s'\\128r'&4&\\128t'&&4\end{matrix}\right)K^{\GL_3} $ if and only if $\left(\begin{matrix}r&s\\t&u\end{matrix}\right)\equiv\left(\begin{matrix}r'&s'\\t'&u'\end{matrix}\right) \bmod 4$.

Therefore there exists a bijection 
\begin{align*}
K^{\GL_3} \left(\begin{matrix}36&1&\\128&4&\\&&4\end{matrix}\right)K^{\GL_3} / K^{\GL_3} &\to \SL_2(\Zz/4\Zz);\\
 \left(\begin{matrix}a&b&c\\128x&r&s\\128y&t&u\end{matrix}\right)\left(\begin{matrix}36&1&\\128&4&\\&&4\end{matrix}\right) K^{\GL_3}   &\mapsto  \left(\begin{matrix}a^{-1}r&\frac{as}{ru-st}\\a^{-1}t&\frac{au}{ru-st}\end{matrix}\right)\ \bmod 4.
 \end{align*}
\end{proof}

Since $F$ is $K_7$-invariant and $\chi$ is unramified, there exists $f\in \Pi^{K^{\GL_3}}$ such that $f(1) = F$ and whose support is contained in $P(\Qq_2)K^{\GL_3}$.
Then, by Equation (\ref{hecke6}) and Lemma \ref{lemma21}, we obtain
\begin{align*}
T^{\GL_3}_{\left(\begin{smallmatrix}36&1&\\128&4&\\&&4\end{smallmatrix}\right)}f&=8f,\\
 \sum_{\substack{r,s,t,u\in\{-1,0,1,2\}\\ ru-st \equiv 1 \bmod 4}}f(\left(\begin{matrix}4&u&-s\\128r&4&0\\128t&0&4\end{matrix}\right)) &=8f(1).
\end{align*}
Since $\supp f \subset P(\Qq_2)K^{\GL_3}$, if $t\not=0$ then $f(\left(\begin{matrix}4&u&-s\\128r&4&0\\128t&0&4\end{matrix}\right))=0$. So we have
\begin{align*}
 8F &= 8f(1)\\
 &=\sum_{\substack{r,s,t,u\in\{-1,0,1,2\}\\ ru-st \equiv 1 \bmod 4}}f(\left(\begin{matrix}4&u&-s\\128r&4&0\\128t&0&4\end{matrix}\right)) \\
 &= \sum_{r\in\{-1,1\}, s\in\{-1,0,1,2\}}f(\left(\begin{matrix}4&r&-s\\128r&4&0\\0&0&4\end{matrix}\right)) \\
 &= 4\left(\pi(\left(\begin{matrix}1&1/4\\32&1\end{matrix}\right))F + \pi(\left(\begin{matrix}1&-1/4\\-32&1\end{matrix}\right))F \right).
 \end{align*}
 Assume $\alpha =\frac{1}{8}\left(\begin{matrix}&1\\2&\end{matrix}\right)$. Then we obtain
 \begin{align*}
2 &= 2F(\left(\begin{matrix}8&\\&1\end{matrix}\right))\\
&=\pi(\left(\begin{matrix}1&1/4\\32&1\end{matrix}\right))F(\left(\begin{matrix}8&\\&1\end{matrix}\right)) + \pi(\left(\begin{matrix}1&-1/4\\-32&1\end{matrix}\right))F(\left(\begin{matrix}8&\\&1\end{matrix}\right))\\
&= F(\left(\begin{matrix}8&\\&1\end{matrix}\right)\left(\begin{matrix}1&1/4\\32&1\end{matrix}\right)) + F(\left(\begin{matrix}8&\\&1\end{matrix}\right)\left(\begin{matrix}1&-1/4\\-32&1\end{matrix}\right)) \\
&=\psi_\alpha(\left(\begin{matrix}1&2\\4&1\end{matrix}\right)) F(\left(\begin{matrix}8&\\&1\end{matrix}\right))+ \psi_\alpha(\left(\begin{matrix}1&-2\\-4&1\end{matrix}\right))F(\left(\begin{matrix}8&\\&1\end{matrix}\right)) \\
&=\psi(\tr\left(\frac{1}{8}\left(\begin{matrix}&1\\2&\end{matrix}\right)\left(\begin{matrix}1&2\\4&1\end{matrix}\right)\right)) F(\left(\begin{matrix}8&\\&1\end{matrix}\right))+ \psi(\left(\tr\frac{1}{8}\left(\begin{matrix}&1\\2&\end{matrix}\right)\left(\begin{matrix}1&-2\\-4&1\end{matrix}\right)\right))F(\left(\begin{matrix}8&\\&1\end{matrix}\right)) \\
&=\psi(1) + \psi(-1) = -2,
\end{align*}
which is a contradiction. So we have $\alpha =\frac{1}{8}\left(\begin{matrix}&1\\-2&\end{matrix}\right)$.

Next, we determine $\Lambda$. Since $J_\alpha / U_\mathfrak{U}^3 \cong E^\times / E^\times \cap U_\mathfrak{U}^3$ is generated by $\left(\begin{matrix}1&1\\-2&1\end{matrix}\right)$, $\left(\begin{matrix}-1&\\&-1\end{matrix}\right)$ and $\left(\begin{matrix}&1\\-2&\end{matrix}\right)$, we have only to determine $\Lambda(\left(\begin{matrix}1&1\\-2&1\end{matrix}\right))$ and $\Lambda(\left(\begin{matrix}&1\\-2&\end{matrix}\right))$.

We calculate $\Lambda(\left(\begin{matrix}1&1\\-2&1\end{matrix}\right))$.
First, since $\Lambda|_{\Zz_2^\times}$ is trivial, we have
\begin{align*}
F(\left(\begin{matrix}8&1\\16&1\end{matrix}\right)) &= F(\frac{-1}{3}\left(\begin{matrix}1&1\\-2&1\end{matrix}\right)\left(\begin{matrix}1&\\-4&-3\end{matrix}\right)\left(\begin{matrix}8&\\&1\end{matrix}\right))\\
&= \Lambda(\left(\begin{matrix}1&1\\-2&1\end{matrix}\right))\psi_\alpha(\left(\begin{matrix}1&\\-4&-3\end{matrix}\right))F(\left(\begin{matrix}8&\\&1\end{matrix}\right)) \\
&= \Lambda(\left(\begin{matrix}1&1\\-2&1\end{matrix}\right))\psi\left(-\frac{1}{2}\right)F(\left(\begin{matrix}8&\\&1\end{matrix}\right)).
\end{align*}
Therefore, we have 
\begin{align*}
\Lambda(\left(\begin{matrix}1&1\\-2&1\end{matrix}\right))\ = \psi\left(\frac{1}{2}\right)F(\left(\begin{matrix}8&1\\16&1\end{matrix}\right)).\stepcounter{equation}\tag{\theequation}\label{lambdavalue}
\end{align*}

By a similar argument as the proof of Lemma \ref{lemma21}, we have
\begin{align*}
K^{\GL_3} \left(\begin{matrix}8&1&\\128&8&\\&&8\end{matrix}\right)K^{\GL_3} = \coprod_{\substack{r,s,t,u\in\{0, \ldots, 7\}\\ ru-st \equiv 1 \bmod 8}}\left(\begin{matrix}8&u&-s\\128r&8&0\\128t&0&8\end{matrix}\right)K^{\GL_3} .
\end{align*}
Therefore, by Equation (\ref{hecke7}), we have
\begin{align*}
32f &=T^{\GL_3}_{\left(\begin{smallmatrix}8&1&\\128&8&\\&&8\end{smallmatrix}\right)}f,\\
32f(1) &=\sum_{\substack{r,s,t,u\in\{0, \ldots, 7\}\\ ru-st \equiv 1 \bmod 8}}\Pi(\left(\begin{matrix}8&u&-s\\128r&8&0\\128t&0&8\end{matrix}\right))f(1),\\
 32F&= \sum_{r\in\{1,3,5,7\}, s\in\{0, \ldots, 7\}}f(\left(\begin{matrix}8&1/r&-s\\128r&8&0\\0&0&8\end{matrix}\right)) \\
 &= 8 \sum_{r\in\{1,3,5,7\}}\pi(\left(\begin{matrix}1&1/8r\\16r&1\end{matrix}\right))F.
 \end{align*}
 So we have
 \begin{align*}
 4&=4F(\left(\begin{matrix}8&\\&1\end{matrix}\right))\\
 &=  \sum_{r\in\{1,3,5,7\}}\pi(\left(\begin{matrix}1&1/8r\\16r&1\end{matrix}\right))F(\left(\begin{matrix}8&\\&1\end{matrix}\right))\\
 &=  \sum_{r\in\{1,3,5,7\}}F(\left(\begin{matrix}8&\\&1\end{matrix}\right)\left(\begin{matrix}1&1/8r\\16r&1\end{matrix}\right))\\
&=  \sum_{r\in\{1,3,5,7\}}F(\left(\begin{matrix}2r-1&\\&2r-1\end{matrix}\right)\left(\begin{matrix}\frac{-r+2}{r(2r-1)}&\frac{r-1}{r(2r-1)}\\\frac{2(1-r)}{2r-1}&1\end{matrix}\right)\left(\begin{matrix}8&1\\16&1\end{matrix}\right))\\
&=  \sum_{r\in\{1,3,5,7\}}F(\left(\begin{matrix}\frac{-r+2}{r(2r-1)}&\frac{r-1}{r(2r-1)}\\\frac{2(1-r)}{2r-1}&1\end{matrix}\right)\left(\begin{matrix}8&1\\16&1\end{matrix}\right)).
\end{align*}
The last equation follows from the fact that $r\in\Zz_2^\times$ and so $\Lambda(\left(\begin{matrix}2r-1&\\&2r-1\end{matrix}\right)) = \omega_\pi(2r-1) = 1$.

Since $r\in \Zz_2^\times$ and  $\frac{-r+2}{r(2r-1)}-1 = \frac{-r+2-r(2r-1)}{r(2r-1)} = -\frac{2(r^2-1)}{r(2r-1)} \equiv 0 \bmod 4$, we have $\left(\begin{matrix}\frac{-r+2}{r(2r-1)}&\frac{r-1}{r(2r-1)}\\\frac{2(1-r)}{2r-1}&1\end{matrix}\right) \in U_\mathfrak{U}^3$. So we have 
\begin{align*}
4&=  \sum_{r\in\{1,3,5,7\}}F(\left(\begin{matrix}\frac{-r+2}{r(2r-1)}&\frac{r-1}{r(2r-1)}\\\frac{2(1-r)}{2r-1}&1\end{matrix}\right)\left(\begin{matrix}8&1\\16&1\end{matrix}\right))\\
&=\sum_{r\in\{1,3,5,7\}}\psi_\alpha(\left(\begin{matrix}\frac{-r+2}{r(2r-1)}&\frac{r-1}{r(2r-1)}\\\frac{2(1-r)}{2r-1}&1\end{matrix}\right))F(\left(\begin{matrix}8&1\\16&1\end{matrix}\right))\\
&=\sum_{r\in\{1,3,5,7\}}\psi(\tr\frac{1}{8}\left(\begin{matrix}&1\\-2\end{matrix}\right)\left(\left(\begin{matrix}\frac{-r+2}{r(2r-1)}&\frac{r-1}{r(2r-1)}\\\frac{2(1-r)}{2r-1}&1\end{matrix}\right) -1\right))F(\left(\begin{matrix}8&1\\16&1\end{matrix}\right))\\
&=\sum_{r\in\{1,3,5,7\}}\psi(\frac{1}{8}\left(\frac{2(1-r)}{2r-1} - \frac{2(r-1)}{r(2r-1)}\right))F(\left(\begin{matrix}8&1\\16&1\end{matrix}\right))\\
&=\sum_{r\in\{1,3,5,7\}} \psi(\frac{1}{8}\frac{2(r^2-1)}{r(2r-1)} )F(\left(\begin{matrix}8&1\\16&1\end{matrix}\right))\\
&=4F(\left(\begin{matrix}8&1\\16&1\end{matrix}\right)),
\end{align*}
since $r^2-1\equiv 0 \bmod 8$ for $r\in\Zz_2^\times$.
Therefore, we have $F(\left(\begin{matrix}8&1\\16&1\end{matrix}\right))=1$, and by Equation (\ref{lambdavalue}), we obtain $\Lambda(\left(\begin{matrix}1&1\\-2&1\end{matrix}\right)) = \psi\left(\frac{1}{2}\right)$.
 
 Finally we calculate $\Lambda(\left(\begin{matrix}&1\\-2&\end{matrix}\right))$.
\begin{lemma}\label{lemma45} We have 
\[
K^{\GL_3}\left(\begin{matrix}&1&\\-128&&\\&&1\end{matrix}\right)K^{\GL_3} = \coprod_{i=0}^{127}\left(\begin{matrix}&1&\\-128&&i\\&&1\end{matrix}\right)K^{\GL_3}  \amalg \coprod_{i=0}^{63}\left(\begin{matrix}&1&\\&&1\\-128&&2i\end{matrix}\right)K^{\GL_3}  .
\]
\end{lemma}
\begin{proof}
By Equation (\ref{decompofK7}), for any $\left(\begin{matrix}a&b&c\\128x&r&s\\128y&t&u\end{matrix}\right) \in K_7$, we have
\begin{align*}
\left(\begin{matrix}a&b&c\\128x&r&s\\128y&t&u\end{matrix}\right) &= \left(\begin{matrix}1&&\\&r&s\\&t&u\end{matrix}\right)\left(\begin{matrix}a&&\\&1&\\&&1\end{matrix}\right)\left(\begin{matrix}1&a^{-1}b&a^{-1}c\\ \frac{128(ux-sy)}{ru-st}&1&\\ \frac{128(-tx+ry)}{ru-st}&&1\end{matrix}\right). 
\end{align*}
For $w$, $x$, $y$, $z\in\Zz_2$, we have
\[
\left(\begin{matrix}1&w&x\\ 128y&1&\\128 z&&1\end{matrix}\right)\left(\begin{matrix}&1&\\-128&&\\&&1\end{matrix}\right)=\left(\begin{matrix}&1&\\-128&&\\&&1\end{matrix}\right)\left(\begin{matrix}1 &-y&0\\-128w&1&x\\0&128z&1\end{matrix}\right),
\]
and for $a\in\Zz_2^\times$ we have
\begin{align*}
 \left(\begin{matrix}a&&\\ &1&\\ &&1\end{matrix}\right) \left(\begin{matrix}&1&\\-128&&\\&&1\end{matrix}\right) &=\left(\begin{matrix}&1&\\-128&&\\&&1\end{matrix}\right) \left(\begin{matrix} 1&&\\&a&\\&&1\end{matrix}\right).
 \end{align*}
For $\left(\begin{matrix}r&s\\t&u\end{matrix}\right)\in\GL_2(\Zz_2)$ with $u \in \Zz_2^\times$, we have
\begin{align*}
 \left(\begin{matrix}1&&\\ &r&s\\ &t&u\end{matrix}\right) \left(\begin{matrix}&1&\\-128&&\\&&1\end{matrix}\right) &=  \left(\begin{matrix}&1&\\-128&&\frac{s}{u}\\&&1\end{matrix}\right) \left(\begin{matrix}\frac{ru-st}{u}&&\\ &1&\\ -128t&&u\end{matrix}\right),
 \end{align*}
and for $\left(\begin{matrix}r&s\\t&u\end{matrix}\right)\in\GL_2(\Zz_2)$ with $u \in 2\Zz_2$, we have $ s \in \Zz_2^\times $ and 
\begin{align*}
 \left(\begin{matrix}1&&\\ &r&s\\ &t&u\end{matrix}\right) \left(\begin{matrix}&1&\\-128&&\\&&1\end{matrix}\right) &=  \left(\begin{matrix}&1&\\&&1\\-128&&\frac{u}{s}\end{matrix}\right) \left(\begin{matrix}\frac{-ru+st}{s}&&\\ &1&\\ -128r&&s\end{matrix}\right).
 \end{align*}
 Also, for $i$, $j\in\Zz_2$, we have 
 \begin{align*}
 \left(\begin{matrix}&1&\\-128&&i + 128j\\&&1\end{matrix}\right) =  \left(\begin{matrix}&1&\\-128&&i \\&&1\end{matrix}\right)\left(\begin{matrix}1&&-j\\&1&\\&&1\end{matrix}\right),\\
 \left(\begin{matrix}&1&\\&&1\\-128&&2i+128j\end{matrix}\right) =  \left(\begin{matrix}&1&\\&&1\\-128&&2i\end{matrix}\right) \left(\begin{matrix}1&&-j\\&1&\\&&1\end{matrix}\right).
 \end{align*}
 Therefore we have
\begin{align*}
K^{\GL_3}\left(\begin{matrix}&1&\\-128&&\\&&1\end{matrix}\right)K^{\GL_3} = \bigcup_{i=0}^{127}\left(\begin{matrix}&1&\\-128&&i\\&&1\end{matrix}\right)K^{\GL_3}  \cup \bigcup_{i=0}^{63}\left(\begin{matrix}&1&\\&&1\\-128&&2i\end{matrix}\right)K^{\GL_3} .
\end{align*}
The union on the right hand side is disjoint since for $i$, $j \in \Zz_2$ we have 
 \begin{align*}
 \left(\begin{matrix}&1&\\-128&&i \\&&1\end{matrix}\right)^{-1} \left(\begin{matrix}&1&\\-128&&j \\&&1\end{matrix}\right)  =  \left(\begin{matrix}1&&\frac{i-j}{128}\\&1& \\&&1\end{matrix}\right),\\
 \left(\begin{matrix}&1&\\&&1\\-128&&2i \end{matrix}\right)^{-1} \left(\begin{matrix}&1& \\&&1\\-128&&2j\end{matrix}\right)  =  \left(\begin{matrix}1&&\frac{i-j}{64}\\&1& \\&&1\end{matrix}\right),\\
 \left(\begin{matrix}&1&\\-128&&i \\&&1\end{matrix}\right)^{-1}  \left(\begin{matrix}&1&\\&&1\\-128&&2j\end{matrix}\right)=  \left(\begin{matrix}-i&&\frac{2ij-1}{128}\\&1& \\-128&&2j\end{matrix}\right).
 \end{align*}
 Therefore, this lemma holds.
\end{proof}

By Equation (\ref{hecke5}), we have
\begin{align*}
T^{\GL_3}_{\left(\begin{smallmatrix}&1&\\-128&&\\&&1\end{smallmatrix}\right)}f&=(8-8\sqrt{-1})f.
\end{align*}
So, by Lemma \ref{lemma45}, we have
\begin{align*}
(8-8\sqrt{-1})F&=(8-8\sqrt{-1})f(1)\\
&=T^{\GL_3}_{\left(\begin{smallmatrix}&1&\\-128&&\\&&1\end{smallmatrix}\right)}f(1)\\
&=\sum_{i=0}^{127}\Pi(\left(\begin{matrix}&1&\\-128&&i\\&&1\end{matrix}\right)) f(1)+ \sum_{i=0}^{63}\Pi(\left(\begin{matrix}&1&\\&&1\\-128&&2i\end{matrix}\right)) f(1) \\
&=\sum_{i=0}^{127}f(\left(\begin{matrix}&1&\\-128&&i\\&&1\end{matrix}\right)) + \sum_{i=0}^{63}f(\left(\begin{matrix}&1&\\&&1\\-128&&2i\end{matrix}\right)).
\end{align*}
Note that $f(\left(\begin{matrix}&1&\\&&1\\-128&&2i\end{matrix}\right)) =0$ for all $i\in\Zz$ since $\left(\begin{matrix}&1&\\&&1\\-128&&2i\end{matrix}\right)\notin P(\Qq_2)K_7$. Therefore we have
\begin{align*}
(8-8\sqrt{-1})F&=128\pi(\left(\begin{matrix}&1\\-128&\end{matrix}\right))f(1),  \\
\frac{1-\sqrt{-1}}{16}F&=\pi(\left(\begin{matrix}&1\\-128&\end{matrix}\right))F.
\end{align*}
Then we have
\begin{align*}
\frac{1-\sqrt{-1}}{16} &= \frac{1-\sqrt{-1}}{16}F(\left(\begin{matrix}8&\\&1\end{matrix}\right))\\
&=\left(\pi(\left(\begin{matrix}&1\\-128&\end{matrix}\right))F\right)(\left(\begin{matrix}8&\\&1\end{matrix}\right)) \\
&=F(\left(\begin{matrix}8&\\&1\end{matrix}\right)\left(\begin{matrix}&1\\-128&\end{matrix}\right))\\
&=F(\left(\begin{matrix}&8\\-16&\end{matrix}\right)\left(\begin{matrix}8&\\&1\end{matrix}\right))\\
&=\Lambda(\left(\begin{matrix}&8\\-16&\end{matrix}\right))F(\left(\begin{matrix}8&\\&1\end{matrix}\right)).
\end{align*}
Since $\Lambda|_{\Qq_2^\times} = \omega_\pi = \chi^{-1}$, we have 
\begin{align*}
\frac{1-\sqrt{-1}}{16}&=\Lambda(\left(\begin{matrix}&8\\-16&\end{matrix}\right))F(\left(\begin{matrix}8&\\&1\end{matrix}\right))\\
&=\chi(8)^{-1}\Lambda(\left(\begin{matrix}&1\\-2&\end{matrix}\right))\\
&=- \frac{\sqrt{-1}}{8}\Lambda(\left(\begin{matrix}&1\\-2&\end{matrix}\right)).
\end{align*}
So we obtain $\Lambda(\left(\begin{matrix}&1\\-2&\end{matrix}\right)) =\frac{1+\sqrt{-1}}{2}$.

Then, we have the following theorem:
\begin{theo}We have
\[\Pi = \Ind_P^{\GL_3(\Qq_2)}(\pi\boxtimes \chi),\] 
where $\chi$ is the unramified character of $\Qq_2^\times$ with $\chi(2) = -2\sqrt{-1}$ and $\pi$ is the supercuspidal representation of $\GL_2(\Qq_2)$ given by 
\[
\pi = \cInd_{J_\alpha}^{\GL_2(\Qq_2)} \Lambda.
\]
Here $\alpha =\frac{1}{8} \left(\begin{matrix}&1\\-2&\end{matrix}\right)$ and $\Lambda$ is the character of $J_\alpha$ determined by 
\[
\begin{cases}
\Lambda(\left(\begin{matrix}a&\\&a\end{matrix}\right)) = \chi^{-1}(a) &a\in\Qq_2^\times,\\
\Lambda(\left(\begin{matrix}&1\\-2&\end{matrix}\right)) = \frac{1+\sqrt{-1}}{2},\\
\Lambda(\left(\begin{matrix}1&1\\-2&1\end{matrix}\right)) = \psi\left(\frac{1}{2}\right),\\
\Lambda|_{U_\mathfrak{U}^3} = \psi_\alpha.
\end{cases}
\]
\end{theo}

 \bibliography{local_2_adic_component_of_a_gl3_representation}
\bibliographystyle{my_amsplain}

\end{document}